\documentclass[11pt, reqno]{amsart}
\usepackage{amssymb,amsmath,amscd, amsfonts,latexsym}
\usepackage[top = 3cm, bottom = 3cm, left = 3cm, right = 3cm]{geometry}
\usepackage{hyperref}

\newtheorem{theorem}{Theorem}

\newtheorem{lemma}{Lemma}
\newtheorem{proposition}{Proposition}
\newtheorem{conjecture}{Conjecture}
\theoremstyle{remark}

\usepackage[colorinlistoftodos,prependcaption,textsize=tiny]{todonotes}

\title[\title{On the iterates of the Ramanujan $\tau$ function}]{On the prime factors of the iterates of the Ramanujan $\tau$--function}

\author[F. Luca, S. Mabaso and P. St\u anic\u a]{Florian Luca, Sibusiso Mabaso and Pantelimon St\u anic\u a}



\address{Florian Luca \newline
         \indent School of Mathematics, University of the Witwatersrand, \newline
        \indent Private Bag X3, Wits 2050,\newline
         \indent Johannesburg, South Africa}
         
\address{Research Group in Algebraic Structures and Applications, King Abdulaziz University,\newline 
         \indent Jeddah, Saudi Arabia}
         
\address{Centro de Ciencias Matem\'aticas, UNAM,\newline
    \indent Morelia, Mexico}                  

\email{Florian.Luca@wits.ac.za}

\address{Sibusiso Mabaso \newline
          \indent Mangosuthu University of Technology, \newline
        \indent 511 Griffiths Mxenge Hwy,\newline
        \indent Umlazi, Durban, 4301, South Africa}

\email{ASmabaso@mut.ac.za}         

\address{Pantelimon St\u anic\u a\newline
              \indent Naval Postgraduate School, \newline
              \indent Applied Mathematics Department,\newline
              \indent Monterey, CA 93943--5216, USA}
              
\email{pstanica@nps.edu}                      

\begin{document}

% ***********************************************************************************

\begin{abstract}
In this paper, for a positive integer $n\ge 1$, we look at the size and prime factors of the iterates of the Ramanujan $\tau$ function applied to $n$. 
\end{abstract}

\maketitle
% ***********************************************************************************

\section{Introduction}

The Ramanujan $\tau$-function $\tau(n)$ is given as the coefficient of $q^n$ in the expansion
$$
q\left(\prod_{k=1}^{\infty} (1-q^k)\right)^{24}=\sum_{n\ge 1} \tau(n)q^n,\quad {\text{\rm where}}\quad |q|<1.
$$
It is well-known that $\tau$ is a multiplicative function. That is, $\tau(1)=1$ and $\tau(mn)=\tau(m)\tau(n)$ holds for all coprime positive integers $m,n$. Further, 
\begin{equation}
\label{eq:linrec}
\tau(p^{a+2})=\tau(p)\tau(p^{a+1})-p^{11}\tau(p^a)\quad {\text{\rm holds~for~all~primes}}\quad p\ge 2\quad {\text{\rm and~integers}}\quad a\ge 0.
\end{equation}
In addition, $|\tau(p)|<2p^{11/2}$ holds for all primes $p$. In particular, 
$$
\tau(p^a)=\frac{\alpha_p^{a+1}-\beta_p^{a+1}}{\alpha_p-\beta_p}\quad {\text{\rm holds~for~all}}\quad a\ge 1,
$$
where $\alpha_p,\beta_p$ are the roots of the quadratic polynomial $x^2-\tau(p)x+p^{11}$. Since its discriminant $(\alpha_p-\beta_p)^2=\tau(p)^2-4p^{11}$ is negative, it follows that $\alpha_p,\beta_p$ are 
complex conjugates. 
 
In this paper, we look at the dynamical system obtained by iteratively applying $\tau$ to a positive integer $n$. There are two obstructions to doing so. The first obvious one is $\tau(n)$ is negative for some $n$. For example, 
$\tau(2)=-24$. To deal with this, we extend $\tau$ to negative numbers by putting $\tau(n):=\tau(|n|)$ for any nonzero integer $n$. A second more subtle obstruction appears if $\tau(n)=0$ for some $n$. 
It is a conjecture of Lehmer that $\tau(n)\ne 0$ for all $n$. This has not yet been proved. If it is false, then there exists a prime $p$ such that $\tau(p)=0$. We have the following lemma.

\begin{lemma}
\label{lem:1}
Assume that the Lehmer conjecture is false, namely that $\tau(n)=0$ for some integer $n$. Then for every prime $q$, there exist $a_q$ such that $\tau(\tau(q^{a_q}))=0$.  
\end{lemma}

To deal with this situation, we can also put, by definition $\tau(0):=0$. One obvious question to ask is what happens with $\tau^k(n)$ for positive integers $k$ and $n$. We conjecture that the set 
\begin{equation}
\label{eq:Orb}
{\rm Orb}_{\tau}(n):=\{\tau^k(n): k\ge 0\}
\end{equation}
is infinite for all $n>1$. In what follows we give some support to this conjecture. For a positive integer $m$ let $P(m)$ be the largest prime factor of $m$.

\begin{proposition}
\label{prop:1}
Assume the Lehmer conjecture holds. Then for  integers $k\ge 1$ and $n$ even, we have $P(\tau^{(k)}(n))\ge 3^{k-1}+2$. 
\end{proposition}

In particular, if the set ${\text{\rm Orb}}_{\tau}(n)$ contains an even element and the Lehmer conjecture holds, then ${\text{\rm Orb}}_{\tau}(n)$ is infinite. It remains to look at the situation when 
${\text{\rm Orb}}_{\tau}(n)$ contains only odd numbers. The smallest odd number (in absolute value) is $1$. We have the following proposition.

\begin{proposition}
\label{prop:2}
If $|\tau(n)|=1$, then $n=1$. 
\end{proposition}

From now on, we assume that $n>1$, so $|\tau(n)|>1$ and that $\tau(n)$ is odd. It is then well-known that $n$ is a perfect square. Computations suggest that in this case $|\tau(n)|$ is much larger than $n$. If this was true for all perfect squares $n>1$, then in case ${\text{\rm Orb}}_{\tau}(n)$ contains only odd numbers,  the numbers 
$$
|n|,\quad |\tau(n)|,\quad |\tau^{(2)}(n)|,\quad,\ldots
$$
would form  a strictly increasing sequence, so in particular ${\text{\rm Orb}}_{\tau}(n)$ would also be infinite. We cannot prove that this is indeed so,  but we can almost prove it (up to finitely many putative exceptions) under the {\it abc}-conjecture. 

\begin{proposition}
\label{prop:3}
Assume the $abc$-conjecture and the Lehmer conjecture. There exists $n_0$ such that if $n>n_0$ is a perfect square, then $|\tau(n)|>n$. In particular, if ${\text{\rm Orb}}_{\tau}(n)$ contains only odd numbers, then 
${\text{\rm Orb}}_{\tau}(n)$ is infinite. 
\end{proposition}

Finally, let $n$ be such that ${\text{\rm Orb}}_{\tau}(n)$ has at least cardinality $k+1$ and assume the Lehmer conjecture. Put
$$
{\text{\rm Orb}}_{\tau}(n,k):=\{|n|, |\tau(n)|,\ldots,|\tau^{(k)}(n)|\},
$$
and assume that the elements in the above set are all distinct and nonzero. We ask whether we can say something about the prime factors of the above numbers. Let 
$$
P({\text{\rm Orb}}_{\tau}(n,k)):=\max\{P(m), m\in {\text{\rm Orb}}_{\tau}(n,k)\}.
$$
We have the following result. 

\begin{proposition}
\label{prop:4}
Assume $k\ge 1$ is an integer. If the Lehmer conjecture holds and ${\text{\rm Orb}}_{\tau}(n,k)$ has cardinality $k+1$, then 
\begin{equation}
\label{eq:Orbnk}
P({\text{\rm Orb}}_{\tau}(n,k))>\log(k/2).
\end{equation}
\end{proposition}

Finally, let ${\mathcal S}=\{p_1,\ldots,p_s\}$ be a finite set of odd primes arranged increasingly and let $P=p_s$ be the largest. We ask whether it is possible to compute the cardinality 
of the set of $n$ such that 
\begin{equation}
\label{eq:Sunit}
\tau(n)=\pm p_1^{a_1}\cdots p_s^{a_s}\quad {\text{\rm for~some~integer~exponents}}\quad a_1,\ldots,a_s.
\end{equation}
By the multiplicativity of the function $\tau$, it suffices to compute the number of solutions of the form $p^a$ for some odd prime $p$ and even positive integer $a$. If this number is denoted by $\Omega$, then clearly the total number of solutions of \eqref{eq:Sunit} is at most $2^{\Omega}$. 

\begin{proposition}
\label{prop:6}
The number of solutions $n=p^a$ with a prime $p$ and a positive even integer $a$ to the equation \eqref{eq:Sunit} is at most $\max\{P,11\}^{6500(s+4)}$. 
\end{proposition}

Taking $S=\{3,5,7\}$, we have $s=3,~P=7$ so equation \eqref{eq:Sunit} has at most $11^{45500}$ solutions of the form $n=p^a$ for an odd prime $p$ and a positive even integer $a$. We end up with a computational result showing that in fact there is no such solution. 

\begin{proposition}
\label{prop:5}
There is no $n>1$ such that $\tau(n)$ is odd and  $P(\tau(n))\le 10$. 
\end{proposition}

This last result shows that in Proposition \ref{prop:6} the $\max\{P,11\}$ can be replaced by $P$. Indeed, either $P\ge 11$, in which case the maximum is $P$, or $P\le 10$, in which case there is no solution to equation \eqref{eq:Sunit} so the stated inequality holds anyway. 

Our paper is much inspired by \cite{MoSmi} where the authors have studied the set of positive integers $n$ such that $n\mid 
\tau(n)$ and on the way they derived congruences for the $\tau$-function such as $\tau(1000 n)\equiv 0\pmod {64000}$ which holds for all positive integers $n$. 

\section{Proofs}

\subsection{Proof of Lemma \ref{lem:1}}

Assume $\tau(n)=0$ for some $n$. Then there is a prime $p$ with $\tau(p)=0$. For $q=p$, we take $a_q=2$. Then $\tau(p^2)=(\tau(p))^2-p^{11}=-p^{11}$. Since $11$ is odd, $\tau(p^{11})$ is an integer multiple of
$\tau(p)$, so it is zero. Hence, $\tau(\tau(p^2))=0$. Assume next that $q\ne p$.  In  particular, we may assume that $\tau(q)\ne 0$. 
The sequence $\{\tau(q^a)\}_{a\ge 0}$ is the shift of a Lucas sequence. Namely, recall that by putting $\alpha_q,~\beta_q$ for the roots of the quadratic
\begin{equation}
\label{eq:quad}
x^2-\tau(q)x+q^{11},
\end{equation}
we have
$$
\tau(q^a)=\frac{\alpha_q^{a+1}-\beta_q^{a+1}}{\alpha_q-\beta_q}.
$$
Further, the ratio $\alpha_q/\beta_q$ is not a root of unity. Indeed, assume it is. Then since it is also a quadratic number, it follows that it is a root of unity of orders one of $1,2,3,4,6$. It is not possible that 
$\alpha_q/\beta_q=1$, since the discriminant $\tau(q)^2-4q^{11}$ of the quadratic \eqref{eq:quad} is negative so $\alpha_q$ and $\beta_q$ are complex non-real and they are conjugates. 
It is also not possible that $\alpha_q/\beta_q=-1$, since we have assumed that $\tau(q)=\alpha_q+\beta_q\ne 0$. Thus, the order is one of $3,4,6$. Since 
$$
\tau(q^a)=\frac{\alpha_q^{a+1}-\beta_q^{a+1}}{\alpha_q-\beta_q},
$$
it follows that $\tau(q^a)=0$ for some $a\in \{2,3,5\}$. However, 
\begin{eqnarray*}
\tau(q^2) & = & \frac{\alpha_q^3-\beta_q^3}{\alpha_q-\beta_q}=\alpha_q^2+\alpha_q\beta_q+\beta_q^2=\tau(q)^2-q^{11};\\
 \tau(q^3) & = & \frac{\alpha_q^4-\beta_q^4}{\alpha_q-\beta_q}=(\alpha_q+\beta_q)(\alpha_q^2+\beta_q^2)=\tau(q)(\tau(q)^2-2q^{11});\\
  \tau(q^5) & = & \frac{\alpha_q^6-\beta_q^6}{\alpha_q-\beta_q}=(\alpha_q+\beta_q)(\alpha_q^2+\alpha_q\beta_q+\beta_q^2)(\alpha_q^2-\alpha_q\beta_a+\beta_q^2)\\
 &=&\tau(q)(\tau(q)^2-q^{11})(\tau(q)^2-3q^{11}),
 \end{eqnarray*}
 so we see that if one of the above expressions is zero but $\tau(q)\ne 0$, it follows that one of $\tau(q)^2-bq^{11}$ is zero for some $b\in \{1,2,3\}$. This is impossible for $b=1$ and for $b\in \{2,3\}$ it implies that $q=b$. So, $q\in \{2,3\}$. However, one 
 checks that  $\alpha_2/\beta_2$ is not a root of unity of order $4$ and that $\alpha_3/\beta_3$ is not a root of unity of order $6$.     
 Hence, $\{\tau(q^a)\}_{a\ge 0}$ is {\it almost} a 
 Lucas sequence, except that it might be that $\tau(q)$ and $q^{11}$ are not coprime. Put $d:=\gcd(\tau(q),q^{11})>1$. Thus, $d=q^{\lambda}$. For $q=2$, we have $\tau(2)=-24$, so $d=8$ and $\lambda=3$.
 For $q=3$, we have $\tau(3)=252=2^2\cdot 3^2\cdot 7$, so $d=9$ and $\lambda=2$. 
 For $q\ge 5$, since $|\tau(q)|<2q^{11/2}$, it follows that $q^{\lambda}\le |\tau(q)|\le 2q^{11/2}<q^6$, so $\lambda\in \{1,2,3,4,5\}$. Then writing $(\gamma_q,\delta_q):=(\alpha_q/q^{\lambda},\beta_q/q^{\lambda})$, we have 
 $$
 \tau(q^a)=q^{\lambda a}\left(\frac{\gamma_q^{a+1}-\delta_q^{a+1}}{\lambda_q-\delta_q}\right).
 $$
 Further, $(\lambda_q,\delta_q)$ are the roots of the quadratic $x^2-(\tau(q)/q^{\lambda}) x+q^{11-2\lambda}$ and $\tau(q)/q^{\lambda}$ and $q^{11-2\lambda}$ are coprime and $\gamma_q/\delta_q=\alpha_q/\beta_q$ is not a root of unity. Thus, 
 $$
 \tau(q^a)=q^{\lambda a} u_q(a+1),
 $$
 where $\{u_q(m)\}_{m\ge 0}$ is the Lucas sequence of roots $(\gamma_q,\delta_q)$. Now let  again $p$ be such that $\tau(p)=0$.  Let $b_p$ be the order of appearance of $p$ in the sequence $\{u_q(m)\}_{m\ge 0}$. 
 This is the smallest positive integer $k$ such that $p\mid u_q(k)$, which  exists since $p$ and $q$ are coprime so the last coefficient of the characteristic equation for $\{u_q(m)\}_{m\ge 0}$ 
 which is $q^{11-2\lambda}$ is coprime to $p$. It is known that $b_p$ divides $p-e$, where $e={\displaystyle{\left(\frac{\tau(p)^2-4p^{11}}{q}\right)}}$ 
 and ${\displaystyle{\left(\frac{\bullet}{q}\right)}}$ is the Legendre symbol. Write 
 $$
 u_q(b_p)=p^{\nu_p} m_p,
 $$
 where $\nu_p\ge 1$ and $m_p$ is coprime to $p$. Let 
 $$
 c_p:=\left\{\begin{matrix} b_p & {\text{\rm if}}  & \nu_p\equiv 1\pmod 2; \\
 pb_p & {\text{\rm  if}} &  \nu_p\equiv 0\pmod 2.\end{matrix}\right.  
 $$
 Since 
 $$
 u_q(pb_p)=p^{\nu_p+1} m_p'
 $$
 for some integer $m_p'$ coprime to $p$, it follows that the exponent of $p$ in $u_q(c_p)$ is  exactly $\mu_p:=2\lfloor \nu_p/2\rfloor+1$ so it is odd. Now compute 
 $$
 \tau(q^{c_p-1})=q^{\lambda (c_p-1)} u_q(c_p)=q^{\lambda(c_p-1)} p^{\mu_p} M_p,
 $$
 where $M_p\in \{m_p,m_p'\}$ is coprime to $p$. Thus, by multiplicativity, 
 $$
 \tau(\tau(q^{c_p-1}))=\tau(p^{\mu_p})\tau(q^{\lambda(c_p-1)}|M_p|)
 $$ 
 and $\tau(p^{\mu_p})$ is a multiple of $\tau(p)$ since  
 $\mu_p$ is odd, so in particular $\tau(\tau(q^{c_p-1})=0$. This proves the lemma with $a_p:=c_p-1$. 
\qed

\subsection{Proof of Proposition \ref{prop:1}}

For a prime $p$ we put $\nu_p(m)$ for the exponent of $p$ in the factorization of $m$. We use the fact that $\nu_2(\tau(n))\ge 3\nu_2(n)$ (see~Lemma 2.1 in~\cite{LuMa}). Thus, if $\nu_2(n)\ge 1$, then $\nu_2(\tau(n))\ge 3$ and by induction on $k$ we get that $\nu_2(\tau^{(k-1)}(n))\ge 3^{k-1}$. Write $|\tau^{(k-1)}(n)|=2^a b$, where $a\ge 3^{k-1}$ and $b$ is odd. 
We look at the sequence $\{\tau(2^m)\}_{m\ge 0}$. With the notation from  the proof of Lemma \ref{lem:1}, 
we have $(\alpha_2,\beta_2)=4\left( -3+{\sqrt{-119}} , 3-{\sqrt{-119}}\right)$. Next, $d=\gcd(\tau(2),2^{11})=2^3$, so $(\gamma_2,\delta_2)=\left(-3+{\sqrt{-119}})/2,(-3-{\sqrt{-119}})/2\right)$. Further,
$$
\tau(2^m)=2^{3m} u_2(m+1),\quad {\text{\rm where}}\quad u_2(m+1)=\frac{\gamma_2^{m+1}-\delta_2^{m+1}}{\gamma_2-\delta_2}.
$$
So, for us, 
$$
\tau^{(k)}(n)=\tau(|\tau^{(k-1)}(n)|)=\tau(2^a b)=2^{3a} u_2(a+1)\tau(b),
$$
where $a+1\ge 3^{k-1}+1$. The sequence $\{u_2(m)\}_{m\ge 0}$ is a Lucas sequence. Thus, by a celebrated result of Bilu, Hanrot and Voutier \cite{BHV}, $u_2(m)$ has a primitive prime factor for $m\ge 31$. 
This is a prime $p$ which does not divide $u_2(\ell)$ for any positive integer $\ell<m$ and does not divide the discriminant $-119=-7\times 17$ either. This prime $p$ has the property that $p\equiv \pm 1\pmod m$. In particular, $p\ge m-1$. 
Applying this to our situation, we get that $u_2(a+1)$ is divisible by a prime $p\ge a\ge 3^{k-1}$ provided $a+1\ge 31$. This last inequality holds for $k\ge 5$ since $a\ge 3^{k-1}.$ For smaller values of $k$, we list 
$u_2(m)$ for all $m\in \{2,\ldots,30\}$ and check that $u_2(m)$ has a primitive prime factor for all such values of $m$ (note that $u_2(m)$ is odd for all $m\ge 1$). Thus, $\tau^{(k)}(n)$ is indeed divisible by a prime $p\ge 3^{k-1}$ for all 
$k\ge 1$. If $k\ge 3$, neither  $3^{k-1}$ nor $3^{k-1}+1$ can be primes so $p\ge 3^{k-1}+2$. Thus, it suffices to prove that the inequality $p\ge 3^{k-1}+2$ also holds for $k=1,2$. For $k=1,~2$, since $n$ is even (so, $a\ge 1$), it follows that the desired inequalities hold if $P(\tau(2^a))\ge 3$ for all $a\ge 1$ and $P(\tau(2^a))\ge 5$ for all $a\ge 3$. But these are equivalent to the fact that $P(u_2(a+1))\ge 3$ for all $a\ge 1$ and $P(u_2(a+1))\ge 5$ 
for all $a\ge 3$,  which are consequences of the fact that the odd number $u_2(m)$ has primitive divisors for all $m\ge 2$ together with the fact that $u_2(2)=-3$. 
\qed

\medskip

{\bf Remark.} Stewart \cite{Stew} showed that $P(u_2(n))>n\exp(\log n/(104\log\log n))$ holds once $n>n_0$. With this result, we get that in fact the inequality $P(\tau^{(k)}(n))>3^{k-1}\exp(k/(104\log k))$ holds 
under the same assumptions (that $n$ is even and that the Lehmer conjecture holds) once $k>k_0$ is sufficiently large. 

\subsection{Proof of Proposition \ref{prop:2}} Since $\tau(n)$ is odd, we may assume that $n$ is an odd square. By the multiplicative property of $\tau$, it follows that we can reduce the problem at a prime power $p^{a}\| n$. 
Then $\tau(p^a)=\pm 1$ and $a$ is even. If $a=2$, then $\tau(p^2)=\tau(p)^2-p^{11}$. Thus, we get $x^2-y^{11}=\pm 1$ with $(x,y):=(\tau(p),p)$, and this has no positive integer solutions $(x,y)$ since, from the solution 
of Catalan's problem by Mih\u ailescu \cite{Mih} we know that $3^2$ and $2^3$ are the only consecutive perfect powers. Thus, $a\ge 4$. Further, $p$ and $\tau(p)$ are coprime for if not then $p\mid \tau(p^a)$ for all $a\ge 1$, which is impossible. 
Thus, 
$$
\tau(p^a)=u_p(a+1)=\frac{\alpha_p^{a+1}-\beta_p^{a+1}}{\alpha_p-\beta_p}=\pm 1.
$$
Since $\tau(p)$ and $p$ are coprime it follows that the above equation signals $u_p(a+1)$ as a member of a Lucas sequence without primitive divisors. Note that $a+1\ge 5$ is odd. These are classified in 
Table 1 in \cite{BHV}. If $(\alpha_p,\beta_p)=\varepsilon((u+{\sqrt{v}})/2,(u-{\sqrt{v}})/2)$, where $\varepsilon\in \{\pm 1\}$ and $u_n=(\alpha_p^n-\beta_p^n)/(\alpha_p-\beta_p)$ does not have a primitive divisor for some 
$n\ge 5$, which is odd, then $n\in \{5,7,13\}$. Further, if $b<0$, then for $n=5$ we have $(a,b)\in \{(1,-7),~(2,-40),~(1,-11),~(1,-15),~(12,-76),~(12,-1364)\}$, while if $n=7,13$ then $(a,b)\in \{(1,-7),~(1,-19)\}$.
However, for us, $a=\pm \tau(p)$ and $b=\tau(p)^2-4p^{11}$. Thus, we get a certain number of equations for $p^{11}$ which we check that they have no convenient solution. 
For example, $(a,b)=(1,-7)$ gives 
$\tau(p)=\pm 1$, $-7=\tau(p)^2-4p^{11}=1-4p^{11}$, so $4p^{11}=8$, a contradiction.  
\qed

\subsection{Proof of Proposition \ref{prop:3}}

We start by analysing $\tau(p^a)$ for odd primes $p$ and even exponents $a$. It turns out that if $a$ is sufficiently large, then $|\tau(p^a)|>p^{2a}$.

\begin{lemma}
\label{lem:2}
If $a>10^{16}$ is even, then $|\tau(p^a)|>p^{2a}$.
\end{lemma}

\begin{proof}
We look again at the sequence $\{\tau(p^a)\}_{a\ge 0}$. Let $d:=\gcd(\tau(p),p^{11})$. In the proof of Lemma \ref{lem:1}, we saw that if $d>1$, then $d=p^{\lambda}$ for some $\lambda\in \{1,2,3,4,5\}$. 
Then it is easy to prove by induction on $a$, via the recurrence formula 
$$
\tau(p^{a+2})=\tau(p)\tau(p^{a+1})-p^{11}\tau(p^a),\quad {\text{\rm valid~for~all}}\quad a\ge 0,
$$
that $p^{\lambda a}\mid \tau(p^a)$. Thus, if $\lambda\ge 3$, then $|\tau(p^{a})|\ge p^{3a}>p^{2a}$ holds for all positive integers $a$. If $\lambda=2$, then $p^{2a}\mid \tau(p^a)$ for all $a\ge 1$. Further, putting 
$(\gamma_p,\delta_p):=(\alpha_p/p^{2},\beta_{p}/p^{2})$ we get 
$$
\tau(p^a)=p^{2a} u_{p}(a+1)\quad {\text{\rm for~all}}\quad a\ge 1,
$$
where $u_{p}$ is the Lucas sequence as in Lemma~\ref{lem:1}.
Thus, $|\tau(p^a)|=p^{2a}$ leads to $u_{p}(a+1)=\pm 1$. If $a=2$, we get 
$$\pm 1=u_p(3)=\gamma_p^2+\gamma_p\delta_p+\delta_p^2=(\gamma_p+\delta_p)^2-\gamma_p\delta_p=(\tau(p)/p^2)^2-p^{7},
$$
which leads to the integer solution $(x,y):=(\tau(p)/p^2,p)$ to the Diophantine equation 
$$
x^2-y^7=\pm 1,
$$
which does not exist by Mih\u ailescu's result \cite{Mih}. If $a\ge 4$, we get again that 
$u_p(a+1)$ is a member of a Lucas sequence without primitive divisors and an investigation of Table 1 in \cite{BHV}, as in the proof of Proposition \ref{prop:2}, does not lead to any solutions. 
Thus, we may assume that $\lambda\in \{0,1\}$. Write again $(\gamma_p,\delta_p):=(\alpha_p/p^{\lambda},\beta_p/p^{\lambda})$. Then we have
$$
\tau(p^a)=p^{\lambda a} u_p(a+1)=p^{\lambda a}\left(\frac{\gamma_p^{a+1}-\delta_p^{a+1}}{\gamma_p-\delta_p}\right).
$$
Thus,
$$
\log |\tau(p^{a})|=\lambda a \log p+(a+1)\log|\delta_p|+\log|(\gamma_p/\delta_p)^{a+1}-1|-\log|\gamma_p-\delta_p|.
$$
We need a lower bound for $\log|\tau(p^a)|$. Since $|\gamma_p-\delta_p|\le 2|\delta_p|=2p^{11/2-\lambda}$, we get
$$
\log |\tau(p^a)|\ge \lambda a\log p+a\log |\delta_p|-\log 2+\log|(\gamma_p/\delta_p)^{a+1}-1|\ge 5.5a\log p-\log 2+\log|\eta^{a+1}-1|,
$$
where $\eta:=\gamma_p/\delta_p$. We need a lower bound for $|\eta^{a+1}-1|$. Corollary 4.2 in \cite{BiLu} shows that if we write $D$ and $h(\eta)$ for the degree and logarithmic height of $\eta$, respectively, then 
the inequality 
$$
\log |\eta^{a+1}-1|>-10^{12} D^4(h(\eta)+1)\log(a+2)
$$
holds (since $|\eta|=1$). A better (sharper) inequality can be found in Lemma 5 in \cite{Vou}. For us, $D=2$ and the minimal polynomial of $\gamma_p/\delta_p$ is 
$$
\gamma_p\delta_p(x-\gamma_p/\delta_p)(x-\delta_p/\gamma_p)=(\gamma_p\delta_p)x^2-(\gamma_p^2+\delta_p^2)x+(\gamma_p\delta_p),
$$
and both $\eta$ and its conjugate  ${\overline{\eta}}=\eta^{-1}$ have absolute value $1$, so 
$$h(\eta)=\frac{1}{2}\log(\gamma_p\delta_p)=\left(\frac{11-2\lambda}{2}\right)\log p\le 5.5\log p.$$ Thus, 
$$
\log|\eta^{a+1}-1|\ge -10^{12}\times 2^4 (5.5\log p+1)\log(a+2)\ge -6.5\times 2^4\times 10^{12} \log p \log(a+2).
$$
Hence,
$$
\log |\tau(p^a)|\ge \log p\left(5.5a-6.5\times 2^4\times 10^{12}\log(a+2)-\frac{\log 2}{\log p}\right)>2a\log p,
$$
where the last inequality holds for all $a>10^{16}$.  
\end{proof}

So, we got that $|\tau(p^a)|>p^{2a}$ holds for all primes $p\ge 3$ and all $a>10^{16}$. Here, we did not need the {\it abc}-conjecture. We use the {\it abc}-conjecture to deal with the low range $a\in [2,10^{16}]$.  
Let us recall the $abc$-conjecture. For a nonzero integer $m$ let
$$
N(m):=\prod_{p\mid m} p,
$$
be the {\it algebraic radical} of $m$.

\begin{conjecture}
For every $\varepsilon>0$, there exists a constant $C:=C_{\varepsilon}$ such that for all nonzero coprime integers $a,b,c$ with $a+b=c$, we have 
$$
\max\{|a|,|b|,|c|\}<C_{\varepsilon} N(abc)^{1+\varepsilon}.
$$ 
\end{conjecture}

In fact, we will use the following consequence of it which is Theorem 5 in \cite{Gra}.

\begin{theorem}
\label{thm:Gra}
Assume that the $abc$-conjecture is true. Suppose that $f(x,y)\in {\mathbb Z}[x,y]$ is homogeneous without repeated factors. Fix $\varepsilon>0$. Then for any coprime integers $n,m$ 
$$
N(f(m,n))>c_{f,\varepsilon}\max\{|m|,|n|\}^{{\text{\rm deg}}(f)-2-\varepsilon}.
$$
The constant $c_{f,\varepsilon}$ depends on $f$ and $\varepsilon$. 
\end{theorem}

We next prove the following statement.

\begin{lemma}
The $abc$-conjecture implies that for all even $a\in [2,10^{16}]$ except for $a=6$, the inequality $|\tau(p^a)|>p^{2a}$ holds for $p>P_a$ sufficiently large. For $a=6$, the $abc$-conjecture implies that 
$|\tau(p^6)|>p^9$ holds for all $p>P_6$ sufficiently large.  
\end{lemma}

\begin{proof}
We start with $a=2$. Then $\tau(p^a)=\tau(p)^2-p^{11}=p^{2\lambda}(x_1^2-p^{11-2\lambda})$, where we put again $p^\lambda=\gcd(\tau(p),p^{11})$, and  $x_1:=\tau(p)/p^{\lambda}$. Recall that we only consider the case $\lambda\in \{0,1\}$. 
If $|\tau(p^2)|\le p^4$, we then get that $|x_1^2-p^{11-2\lambda}|\le p^{4-2\lambda}$. Consider the $abc$-equation $a+b=c$, where $a:=x_1^2,~b:=-p^{11-2\lambda}$. Then $|c|\le p^{4-\lambda}$, $a,~b,~c$ are coprime  and 
$\max\{|a|,|b|,|c|\}=|b|=p^{11-2\lambda}$. We get
$$
p^{11-2\lambda}\ll_{\varepsilon} N(abc)^{1+\varepsilon}\ll_{\varepsilon} (|x_1|p^{4-2\lambda})^{1+\varepsilon}\ll_{\varepsilon} (2p^{10.5-3\lambda})^{1+\varepsilon},
$$
where we used the fact that $|x_1|=|\tau(p)|/p^{\lambda}\le 2p^{5.5-\lambda}$. Choosing $\varepsilon:=0.01$, we get that $p\ll 1$. Thus, $p\le P_2$ is bounded for $a=2$. 

Assume next that $a\in [4,10^{16}]$ is even. Then 
$$
|\tau(p^a)|=p^{a\lambda} u_{p}(a+1)=p^{\lambda a} F_a(\gamma_p,\delta_p),
$$
where 
$$
F_a(X,Y)=\frac{X^{a+1}-Y^{a+1}}{X-Y}=X^a+X^{a-1}Y+\cdots+XY^{a-1}+Y^a.
$$
The polynomial $F_a(X,Y)$ is symmetric in $X$ and $Y$ so it is of the form $G_a(S,P)$ for some polynomial $G_a\in {\mathbb Z}[x,y]$, where we put $S:=X+Y,~P:=XY$. In addition, as a polynomial in $S$ it is concentrated only in even monomials. This can be seen by simultaneously changing the signs of $X$ and $Y$ (so, replacing $(X,Y)$ by $(-X,-Y)$). This does not change $F_a(X,Y)$ since $a$ is even and does not change $P$ but changes the sign of $S$. Thus, $G_a(S,P)=G_a(-S,P)$ so $G_a(S,P)=H_a(S^2,P)$ for some polynomial $H_a(x,y)\in {\mathbb Z}[x,y]$. The polynomial $H_a$ is homogenous of degree $a/2$. Let $a=4$. Then one checks that
$$
H_4(S^2,P)=S^4-3PS^2+P^2=(S^2-3P/2)^2-(5/4)P^2.
$$
So, assume that $|\tau(p^4)|\le p^8$. 
Then putting again $x_1:=\tau(p)/p^{\lambda}$, we get
$$
|\tau(p^4)|=p^{4\lambda} |F_4(\gamma_p,\delta_p)|=p^{4\lambda} |H_{4}(x_1^2,p^{11-2\lambda})|=p^{4\lambda}|(x_1^2-3p^{11-2\lambda}/2)^2-(5/4)p^{2(11-2\lambda)}|.
$$
We thus get that
$$
|(2x_1^2-3p^{11-2\lambda})^2-5 p^{2(11-2\lambda)}|\le 4p^{8-4\lambda}.
$$
We apply the $abc$-conjecture to the equation $a+b=c$, where $a:=(2x_1^2-3p^{11-2\lambda})^2$ and $b:=\textcolor{red}{-}5p^{2(11-2\lambda)}$. The greatest common divisor $D$ of these two numbers is either $1$  or $5$ because $p$ does not divide $a$ since $p$ does not divide $x_1$. Thus, applying the $abc$--conjecture to the equation $a_1+b_1=c_1$, where $a_1:=a/D,~b_1:=b/D$, we get
$$
p^{2(11-2\lambda)}\le \max\{|a_1|,|b_1|,|c_1|\}\ll_{\varepsilon} N(abc)^{1+\varepsilon}\le (11p^{11-2\lambda} \times (5p)\times (4p^{8-\lambda}))^{1+\varepsilon}\ll p^{(20-3\lambda)(1+\varepsilon)}
$$
and taking $\varepsilon:=0.01$, we get again that $p\le P_4$ (both when $\lambda=0$ and $\lambda=1$).

Now assume that $a\ge 6$. Then $H_a(x,y)$ has degree $a/2\ge 3$. We apply Theorem \ref{thm:Gra} to it to infer that
$$
|\tau(p^a)|=p^{\lambda a} |F_a(\gamma_p,\delta_p)|=p^{\lambda a} |H_a(x_1^2,p^{11-2\lambda}|\gg_{a,\varepsilon} p^{\lambda a} (p^{11-2\lambda})^{a/2-2-\varepsilon}\gg_{a,\varepsilon} p^{5.5a-11(2+\varepsilon)+2\lambda(2+\varepsilon)}.
$$
We want that $|\tau(p^a)|>p^{2a}$. This will be so if
$$
c_{H_a,\varepsilon} p^{5.5a-11(2+\varepsilon)+2\lambda(1+\varepsilon)}>p^{2a},
$$
where $c_{H_a,\varepsilon}$ is the constant that comes from Theorem \ref{thm:Gra}.   This works for $a\ge 8$ since we can take $\varepsilon:=0.01$, and we see that it is enough that 
$$
p>c_{H_a,0.01}^{-1/(3.5a-(2.01\times 11)}:=P_a,
$$
and the denominator of the exponent of $c_{H_a,0.01}$ is positive for $a\ge 8$. It also works for $a=6$ and $\lambda=1$, but it fails for $a=6$ and $\lambda=0$, since then $3.5a+2\lambda(2+\varepsilon)=21<11(2+\varepsilon)$. However, for $a=6$, we can replace the lower bound $|\tau(p^a)|>p^{2a}$ by $|\tau(p^a)|>p^{1.5 a}$. This works if $4a>11(2+\varepsilon)$ and this is satisfied with $a=6$ and $\varepsilon=0.01$. Thus, for $a=6$, we get
$$
p>c_{H_6,0.01}^{-1/(24-2.01\times 11)}:=P_6,
$$
which is what we wanted. 
\end{proof}
 
 We are now ready to finish  the proof of Proposition \ref{prop:3}. 
 
 \medskip
 
 \noindent
 {\it Proof of Proposition~\textup{\ref{prop:3}}}. For each $a\in [2,10^{16}]$, we let $P_a$ be such that if $|\tau(p^a)|>p^{1.5 a}$ then $p\le P_a$. Let 
 $$
 n_0:=\prod_{\substack{2\le a\le 10^{16}\\ 2\mid a}} P_a^a.
 $$
 Let $n>n_0^3$ be an odd square. Write $n=n_1n_2$ where $\gcd(n_1,n_2)=1$, and $n_1$ is built up of prime powers $p^{a_p}$ such that $|\tau(p^{a_p})|\le p^{1.5a_p}$. Then $a_p\le 10^{16}$ and $p\le P_{a_p}$. 
 In particular, $n_1\le n_0<n^{1/3}$, so $n_2>n^{2/3}$. Thus, since $|\tau(p^{a_p})|>p^{1.5a_p}$ for all prime powers $p^{a_p}$ dividing $n_2$, we get 
 $$
 |\tau(n)|\ge |\tau(n_2)|>n_2^{1.5}> (n^{2/3})^{1.5}=n,
 $$
 which completes the proof.\qed
 
 \subsection{Proof of Proposition \ref{prop:4}}
 
 We may assume that $k>10$ otherwise the left--hand side of~\eqref{eq:Orbnk} is smaller than $2$ and the inequality trivially holds by Proposition~\ref{prop:2}. If among the elements of 
 \begin{equation}
 \label{eq:list}
 L:=\{n,\tau(n),\ldots, \tau^{(\lfloor k/2\rfloor)}(n)\}
 \end{equation}
 there is an even number, then with $m:=\tau^{(\lfloor k/2\rfloor)}(n)$, we have that $m$ is even and Proposition~\ref{prop:1} shows that 
 $$
 P(\tau^{(k)}(n))=P(\tau^{(k-\lfloor k/2\rfloor)}(m))\ge 3^{k-\lfloor k/2\rfloor-1}+2\ge 3^{k/2-1}+2>0.1\log(k/2)\quad {\text{\rm for}}\quad k>10.
 $$
 Thus, we may assume that all numbers in list $L$ given at \eqref{eq:list} are odd. We look at prime powers $p^a$, where $p^a\| s$ for some number $s$ from list \eqref{eq:list}.
 Let $\omega$ be the number of such primes $p$ and $P$ be the largest. Then $p$'s are odd, $a$'s are even. Further, by the Primitive Divisor Theorem and Table 1 in \cite{BHV}, 
 each of $u_p(a+1)$ for $a\ge 4$, contains a primitive prime factor $q$ which does not divide $u_p(b+1)$ for any even $b<a$. This together with the fact that $|\tau(p^2)|=p^{2\lambda} u_p(3)$ and $|u_p(3)|>1$ is coprime to $p$, shows that for a fixed $p$, there can be at most $\omega$ values of $a$.  This shows that 
 $$
 k/2\le \# L\le \omega^\omega=e^{\omega \log \omega}\le e^{p_{\omega}}<e^P,
 $$
 where we use $p_1<p_2<\cdots$ for the sequence of all prime numbers and the fact that $P> p_{\omega}\ge \omega\log \omega$ (see (3.12) in~\cite{Rosser}). 
 Hence, $P\ge \log(k/2)$, as we wanted. 
 \qed

 \subsection{Proof of Proposition \ref{prop:6}}
 
  For technical reasons, we enlarge $S$ and adjoin the primes $2,5,11$ to it. So, we work with $s$ but in the final answer we need to replace $s$ by $s+3$. 
 We write 
 $$
 {\mathbb Z}_S^*=\{\pm p_1^{a_1}\cdots p_s^{a_s}: a_i\in {\mathbb Z}\}.
 $$
 We want to count the solutions $(p,a)$ where $p$ is an odd prime and $a$ is an even integer to $\tau(p^a)\in {\mathbb Z}_S^*$. Since $\tau(p^a)=p^{\lambda a} u_p(a+1)$
 and $u_p(m)$ has primitive divisors for all even $m\ge 4$ (and $|u_p(2)|>1$ is coprime to $p$), it follows that $a$ can take at most $s$ values and the largest one satisfies $a+1\le P-1$, so  $a+1\le P-2$,
 since $a$ is even and $P$ is odd. Thus, there are at most $P^s$ values of the form $\tau(p^a)$ with $p\in S$. From now on, we assume that $p\not\in S$. In particular, 
 $p\nmid \tau(p)$.  Observe that
 $$
 \tau(p^a)=u_p(a+1)=\frac{\alpha_p^{a+1}-\beta_p^{a+1}}{\alpha_p-\beta_p}.
 $$
 Let $q:=P(a+1)$ and then 
 $$
 \tau(p^a)=\left(\frac{(\alpha_p^{(a+1)/q})^q-(\beta_p^{(a+1)/q})^q}{\alpha_p^{(a+1)/q}-\beta_p^{(a+1)/q}}\right)u_{p}((a+1)/q).
 $$
 Thus, putting $(\alpha_1,\beta_1):=(\alpha_p^{(a+1)/q}, \beta_p^{(a+1)/q})$, we have that
 $$
 \tau(p^a)=F_{q-1}(\alpha_1,\beta_1) u_p((a+1)/q),
 $$
 so we deduce that $F_{q-1}(\alpha_1,\beta_1)\in {\mathbb Z}_S^*$. If $q=3$, then 
 $$
 F_2(\alpha_1,\beta_1)=\alpha_1^2+\alpha_1\beta_1+\beta_1^2=S_1^2-P_1,
 $$
 where $S_1:=\alpha_1+\beta_1$ and $P_1:=\alpha_1\beta_1$. Since $P_1=p^{11(a+1)/q}$,  we get with $(x,y):=(S_1,p^{(a+1)/q})$ that  
 \begin{equation}
 \label{eq:10}
 x^2-y^{11}\in {\mathbb Z}_S^*.
 \end{equation}
 If $q=5$, then
 $$
 F_4(\alpha_1,\beta_1)=(\alpha_1+\beta_1)^4-3(\alpha_1\beta_1)(\alpha_1+\beta_1)^2+\alpha_1\beta_1=(S_1^2-3P_1/2)^2-5P_1^2/4.
$$
Multiplying both sides above by $4\times 5^{10}$, we get 
$$
4\times 5^{10} F_5(\alpha_1,\beta_1)=(5^5(2S_1^2-3P_1))^2-5^{11} P_1^2.
$$
Since $P_1=p^{11(a+1)/q}$, we can take $(x,y):=(5^5\times (2S_1^2-3P_1),5p^{2(a+1)/q})$ and get that
\begin{equation}
\label{eq:11}
4\times 5^{10} F_5(\alpha_1,\beta_1)=x^2-y^{11}\in {\mathbb Z}_S^*.
\end{equation}
Finally, assume that $q\ge 7$. We then have 
$$
F_{q-1}(\alpha_1,\beta_1)=H_{q-1}(S_1^2,P_1),
$$
and $H_{q-1}(x,y)$ is a homogeneous polynomial of degree $(q-1)/2\ge 3$ which is irreducible (the polynomial $H_{q-1}(x,1)$ is of degree $\phi(q)/2=(q-1)/2$ and 
any of its roots spans the maximal real subfield of the cyclotomic field ${\mathbb Q}(e^{2\pi i/q})$). Thus, we get
\begin{equation}
\label{eq:12}
H_{q-1}(x,y)\in {\mathbb Z}_S^*.
\end{equation}
Any solution $(x,y)$ of \eqref{eq:10} or \eqref{eq:12} which is convenient for us must have $y=p^{(a+1)/q}$, which a power of a prime $p$. Since $q$ is known, so is $a$. 
Any solution $(x,y)$ of \eqref{eq:11} which is convenient for us must have $y=5p^{2(a+1)/q}$, so again $p$ and then $a$ are known. Thus, it suffices to count the number of such solutions. 
For \eqref{eq:10} and \eqref{eq:11}, we write 
$$
x^2=y^{11}+A,
$$
where $A=\pm p_1^{a_1}\cdots p_s^{a_s}$. In the case of \eqref{eq:11}, we must also allow factors of $4\times 5^{10}$, which is why we enlarged $S$ to contain $2$ and $5$.  
We may reduce $a_i$ modulo $22$ such as to write them as $a_i=22b_i+r_i$ where $r_i\in \{0,1,\ldots,21\}$. Then we get that putting $z:=\prod_{i=1}^s p_i^{b_i}$, $A':=\prod_{i=1}^s p_i^{r_i}$ 
and $x':=x/z^{11},~y':=y/z^2$, we have 
\begin{equation}
\label{eq:A'}
x'^2=y'^{11}+A'.
\end{equation}
The number of choices of $A'$ is at most $2\times 22^s$. Note that $(x',y')\in {\mathbb Z}_{S}^*$. Further, knowing $y'$ we determine $y$ uniquely since $p$ is not in $S$.  
For each one of these, by Theorem~1 in~\cite{ES}, the number of solutions is at most
$$
7^{11^2(4+9s)} h({\mathbb L})^2,
$$
where ${\mathbb L}$ is some extension of ${\mathbb Q}$ containing three of the roots of $f_{A'}(y):=y^{11}+A'$ and $h({\mathbb L})$ is its class number. The above bound is valid provided that all primes dividing the discriminant of $f_{A'}(y)$ are in $S$, which is why we incorporated $11$ into $S$. We can take ${\mathbb L}$ to
be the field ${\mathbb Q}(A'^{1/11},e^{2\pi i/11})$, where $A'^{1/11}$ is the real $11$th root of $A'$. The absolute values of the discriminant of ${\mathbb K}_1:={\mathbb Q}(A'^{1/11})$ is 
at most $(11|A'|)^{11}$ and of the discriminant of ${\mathbb K}_2:={\mathbb Q}(e^{2\pi i/11})$ is at most $11^{11}$, and the degrees of both fields are at most $11$. So, 
the discriminant $\Delta_{\mathbb L}$ of ${\mathbb L}$ (which is the compositum of ${\mathbb K}_i$, for  $i=1,2$) is in absolute value at most
$$
(11 A')^{11^2}\times 11^{11^2}\le (11^2 A')^{11^2}\le (P^{22s+2})^{11^2}=P^{2\times 11^2(11s+1)}.
$$ 
Thus, 
\begin{equation}
\label{eq:tttt}
|\Delta_{\mathbb L}|<P^{2\times 11^2(11s+1)}.
\end{equation}
From inequalities (3.6) and  (3.7) in \cite{BEG}, we have that 
$$
h({\mathbb L})<5|\Delta_{\mathbb L}|^{1/2}(\log |\Delta_{\mathbb L}|)^{(11^2-1)/2}.
$$
We show that the right--hand side above is smaller that the right--hand side of \eqref{eq:tttt}. Indeed, for that all we have to show is that
$$
5(2\times 11^2(11s+1)\log P)^{(11^2-1)/2}<P^{11^2(11s+1)}.
$$
Since $5<(11^2)^{1/2}$, we have that 
$$
5(2\times 11^2(11s+1)\log P)^{(11^2-1)/2}<(2\times 11^2(11s+1)\log P)^{11^2},
$$
so taking $11^2$ roots, it suffices to show that 
$$
2\times 11^2 (11s+1)\log P<P^{11s+1}.
$$
Since $P>2\log P$, it suffices to show that $P^{11s} >11^2(11s+1)$, which is implied by $P^x>(x+1)^3$ with $x=11s$, and further by, $2^x>(x+1)^3$, and this does indeed hold for all $x\ge 11$. Thus,
$$
h({\mathbb L})\le P^{2\times 11^2(11s+1)},
$$
so the number of solutions of \eqref{eq:A'} is at most 
$$
7^{11^2(4+9s)}\times P^{4\times 11^2(11s+1)}<P^{11^2(4+9s+4+44s)}=P^{11^2(53s+8)},
$$ 
 where we use the fact that $P>7$ (because $11$ has been incorporated into $S$). This is for a fixed $A'$ and there are at most $2\times 22^s=2^{s+1}\times 11^s<(1/2)P^{2s}$ such equations
 for each of $a=2,4$, which gives us a number of possibilities for $p$ at most
 $$
 P^{2s+1}\times P^{11^2(53 s+8)}=P^{6415s+969}.
 $$
 Now we fix $q\ge 7$. In this case, we have the equation
 $$
 H_{q-1}(x,y)\in {\mathbb Z}_S^*.
 $$
This is a Thue--Mahler equation of degree $(q-1)/2\ge 3$ and $H_{q-1}(x,y)$ is irreducible. The number of equivalence classes of solutions $(x,y)\in ({\mathbb Z}_S^*)^2$ 
(where by equivalence classes we mean that $(x,y)\equiv (x',y')$ if $(x,y)=\lambda(x',y')$ for some $\lambda\in {\mathbb Z}_S^*$) is, by a result of Evertse \cite{Ev}, at most
$$
(5\times 10^6(q-1)/2)^s\le (10^7 P)^s<P^{8s}.
$$
Our solutions for which $y$ is a power of a prime $p$ not in $S$ are in inequivalent classes. Thus, the above bound bounds the acceptable number of primes $p$.  
This is for a $q$ fixed and $q\le P$. Each of these determines $p$ and then we need to multiply by another factor of $P$ to account for the number of possibilities for $a$. Thus, the total number
is at most 
$$
P(P^{8s+1}+P^{6435s+969})+P^s<P^{6500(s+1)}.
$$
Now we replace $s$ by $s+4$ (we need to add $3$ for the primes $2,5,11$ and another $1$ for the place at infinity which is always incorporated in both the results from~\cite{Ev} and~\cite{ES}), and we get the desired result. 
\qed

 \subsection{Proof of Proposition \ref{prop:5}}
 
 By the multiplicativity of $\tau$, it suffices to show that there is no prime odd $p$ and even positive integer $a$ such that $P(\tau(p^a))\le 7$. We write again $d:=\gcd(\tau(p),p^{11})=p^{\lambda}$ and 
 $$
 \tau(p^a)=p^{\lambda a} u_{p}(a+1).
 $$
 Then $\{u_p(m)\}_{m\ge 1}$ is a Lucas sequence. If $a\ge 8$, then, since $P(u_p(a+1))\le 7$, we conclude that $u_p(a+1)$ has no primitive divisors. As we saw, there are only finitely many possibilities for 
 $(a,\gamma_p,\delta_p)$ and they can all be read from Table 1 in \cite{BHV} as in the proof of Proposition~\ref{prop:2}. This gives no solutions. Assume next that~$a=6$. Then $P(u_p(7))\le 7$. Hence the largest prime factor of $u_p(7)$ is either at most $5$ (smaller than $7$ and incongruent to $\pm 1$ modulo $7$), or~$7$. But if $7\mid u_p(7)$, then $7$ must divide the discriminant of the sequence (which is the same as the discriminant $(\gamma_p-\delta_p)^2$ of the characteristic polynomial) and such primes do not qualify as primitive primes, so these instances also appear in Table~1 in \cite{BHV}. Similarly, if $a=4$, then $u_p(a+1)=u_p(5)$.  The only primes which can divide $u_p(q)$, with a prime $q$ are either $q$ itself (if $q$ divides the discriminant of the sequence) or primes which are congruent to $\pm 1$ modulo~$q$. 
 Thus, the only possibility is that $u_p(5)$ is a power of $5$, so again $u_p(5)$ has no primitive prime factors and the possibilities can be read off from Table 1 in \cite{BHV}. 
 It remains to consider the case $a=2$. Since each of $\tau(3^2),~\tau(5^2),~\tau(7^2)$ has a prime factor larger than $7$, it follows that $p\ge 11$.   Further, $p\nmid \tau(p)$. Now 
 $$
 \tau(p^2)=\tau(p)^2-p^{11}=\pm 3^a 5^b 7^c.
 $$
 We  use congruences with the Ramanujan function to get information on the exponents $a,b,c$. Here are the congruences that we use. In what follows, $\sigma_k(n)$ is the sum of the $k$th powers of the divisors of $n$.
 \begin{itemize}
 \item[(i)] $\tau(n)\equiv n\sigma_9(n)$ if $n\equiv 0,1,2,4\pmod 7$;
 \item[(ii)] $\tau(n)\equiv n^{-30}\sigma_{71}(n)\pmod {5^3}$ if $n\not\equiv 0\pmod 5$;
 \item[(iii)] $\tau(n)\equiv n^{-610}\sigma_{1231}(n)\pmod {3^6}$ if $n\equiv 1\pmod 3$.
 \end{itemize}
 Assume now that $c\ge 1$. In particular, $p^{11}\equiv \tau(p)^2\pmod 7$, so $p$ is a quadratic residue modulo $7$. Thus, we can apply congruence (i) above to $n=p^2$  to get that 
 $$
 0\equiv \tau(p^2)\pmod 7\equiv p^2\sigma_9(p^2)\pmod 7\equiv p^2(1+p^9+p^{18})\pmod 7\equiv 3p^2\pmod 7,
 $$
 since $p^3\equiv 1\pmod 7$. Hence, $p=7$, a contradiction. Thus, $c=0$. 
 
 Assume next that $b\ge 1$.  We use the congruence (ii) above for $n=p^2$. Note also that since $p^{11}\equiv \tau(p)^2\pmod 5$, we get that $p$ is a quadratic residue modulo $5$. Thus, $p^2\equiv 1\pmod 5$. 
 By (ii) above for $n=p^2$, we get
 $$
 0\equiv p^{-60} \sigma_{71}(p^2)\pmod 5\equiv 1+p^{71}+p^{142}\pmod 5\equiv 2+p\pmod 5,
 $$
 so $p\equiv 3\pmod 5$, which contradicts the fact that $p$ is a quadratic residue modulo $5$. Thus, $b=0$. 
 
 Thus, the only possibility is 
 $$
 \tau(p)^2-p^{11}=\pm 3^a,
 $$
 and $a\ne 0$ since, otherwise we get again the Catalan equation. Hence, $a\ge 1$. Next, since $p^{11}\equiv \tau(p)^2\pmod 3$, it follows that $p\equiv 1\pmod 3$ so $p^3\equiv 1\pmod 9$ and $p^2+p+1\equiv 3\pmod 9$. We now apply (iii) to $n=p^2$ to get that
 $$
 \pm 3^a\equiv \tau(p^2)\pmod {3^6}\equiv p^{-1220}\sigma_{1231}(p^2)\pmod {3^6}\equiv p^{-1220} (1+p^{1231}+p^{2432})\pmod {3^6}.
 $$  
 Since $p^3\equiv 1\pmod 9$, we get that $1+p^{1231}+p^{2432}\equiv 1+p+p^2\pmod 9\equiv 3\pmod 9$. So, in the above congruence, we get
 $$
 \pm 3^a\equiv 3p^{-1220} \pmod 9,
 $$
 which shows that $a=1$ and $\pm 1\equiv p^{-1220}\pmod 3$. Since $p\equiv 1\pmod 3$, we deduce that the sign is $+$. Hence, we get the equation
$$
\tau(p^2)=\tau(p)^2-p^{11}=3.
$$
With $(x,y):=(\tau(p),p)$, we get
\begin{equation}
\label{eq:Barros}
x^2-3=y^{11}.
\end{equation}
This can be reduced to one of several Thue equations. Indeed, the left--hand side factors in the Euclidean ring ${\mathbb Z}[{\sqrt{3}}]$ as $(x-{\sqrt{3}})(x+{\sqrt{3}})$, and the two factors $x+{\sqrt{3}}$ and 
$x-{\sqrt{3}}$ are coprime since $x$ is even. Thus, $y=(a+b{\sqrt{3}})(a-b{\sqrt{3}})$ for some integers $a,~b$ and $(a+b{\sqrt{3}})^{11}$ is associated to one of $x\pm {\sqrt{3}}$. Up to changing $b$ to $-b$, we may assume that
$(a+b{\sqrt{3}})^{11}$ is associated to $x+{\sqrt{3}}$. Thus, 
$$
x+{\sqrt{3}}=(a+b{\sqrt{3}})^{11}\zeta,
$$
 where $\zeta$ is a unit. All units in ${\mathbb Z}[{\sqrt{3}}]$ are of the form $\pm (2+{\sqrt{3}})^k$ for some integer $k$. We may reduce $k$ modulo $11$, so replace $k$ by $11k_0+r$, where $r\in \{0,1,\ldots,10\}$
 and replace $a+b{\sqrt{3}}$ by $(a+b{\sqrt{3}})(2+{\sqrt{3}})^{k_0}$. Thus, we get
 $$
 x+{\sqrt{3}}=\pm (a+b{\sqrt{3}})^{11}(2+{\sqrt{3}})^r\quad {\text{\rm for~some}}\quad r\in \{0,1,\ldots,10\}.
 $$
 Conjugating and eliminating $x$, we get
 $$
 \frac{(a+b{\sqrt{3}})^{11}(2+{\sqrt{3}})^{r}-(a-b{\sqrt{3}})^{11}(2-{\sqrt{3}})^r}{2\sqrt{3}}=\pm 3.
 $$
 These are the Thue equations. For example, for $r=2$, we get
 \begin{eqnarray*}
\pm 3 &=& 4 a^{11}+77 a^{10} b+660 a^9 b^2+3465 a^8 b^3+11880 a^7 b^4+ 29106 a^6 b^5+49896 a^5 b^6\\
& + & 62370 a^4 b^7+53460 a^3 b^8+31185 a^2 b^9+10692 a b^{10}+1701 b^{11}.
\end{eqnarray*}
At any rate, equation \eqref{eq:Barros} (and several others of the type $x^2-D=y^n$ for various $D\in [1,100]$ and various exponents $n\ge 3$) have been solved in Carlos Barros' Ph.D. dissertation \cite{Barros}. The only integer solutions are $(x,y)=(\pm 2,1)$, which are not convenient for us. This completes the proof.  
\qed

\subsection{Comments}

One may wonder what happens if the Lehmer conjecture is false. Then we believe that for most $n$, $\tau^{(2)}(n)=0$, and we provide some heuristics below. Indeed, it is known that there is a constant $c>0$ such that 
on a set of $n$ of asymptotic density $1$, $\tau(n)$ is a multiple of all prime powers $p^a<c_1\log n/\log\log n$. This was done in Proposition 1 in~\cite{GLY} for the $n$th coefficient of the modular
form associated to an elliptic curve over ${\mathbb Q}$ without complex multiplication and the same argument works with the Ramanujan function $\tau(n)$. Let $p$ be such that 
$\tau(p)=0$ and for each prime $q$ let $a_q$ be the minimal positive integer such that $\tau(q^{a_q})$ is divisible by $p$ and the exponent $\nu_p(\tau(q^{a_q}))$ is odd. By the arguments from Lemma \ref{lem:1}, we have that $a_q+1$ is either the order of appearance of $p$ in the Lucas sequence $\{u_q(m)\}_{m\ge 0}$, which we denote by $\ell_p(q)$, or $p\ell_p(q)$. Note that 
$\ell_p(q)$ is at most $q+1$. Let $a_q(n)$ be the exponent of 
$q$ in the factorization of $\tau(n)$. From what we have said, this is at least as large as $\lfloor (1/\log q)\log(c_1\log n/\log\log n)\rfloor$ for most $n$. If $a_q(n)+1$ is an odd multiple of $a_q+1$, 
then $\tau(q^{a_q(n)})$ is an integer multiple of $\tau(q^{a_q})$, so it is  zero. So, we need to look at the situation when $a_q(n)+1$ is not an odd multiple of $a_q+1$. Let us assume that $a_q(n)$ 
behaves like a random large number with respect to being in  a certain residue class modulo $2(a_q+1)$. Thus, let us assume that the  probability that $a_q(n)+1$ is not an odd multiple of $a_q+1$ is $1-1/(2(a_q+1))$ and let us assume that these probabilities are independent as $q$ ranges over small  primes different than $p$. Then the probability that 
$\tau(\tau(n))$ is not zero would therefore be at most 
\begin{equation}
\label{eq:prod}
\prod_{\substack{2\le q<c\log n/\log\log n \\ 2\ne p}} \left(1-\frac{1}{2(a_q+1)}\right),
\end{equation}
and since the sum
$$
\sum_{\substack{q\ge 2\\ q\ne p}} \frac{1}{2(a_q+1)}\ge \frac{1}{2p}\sum_{\substack{q\ge 2\\ q\ne p}} \frac{1}{\ell(q)}\ge \frac{1}{p}\sum_{q} \frac{1}{q+1}
$$
is divergent, it follows that the product shown at \eqref{eq:prod} tends to $0$ with $n$.

\section*{acknowledgements}
We thank Michael Bennett and Samir Siksek for useful correspondence and for pointing out to us the contents of the dissertation \cite{Barros}. F.~L. and P. S. were supported by  grant RTNUM19 from CoEMaSS, Wits, South Africa. S. M. worked on this paper as part of his Ph.D. dissertation. This work was done during a visit of P.~S. at the School of Mathematics of Wits University in Spring 2020. He thanks this Institution for 
hospitality.

 \end{document}